\documentclass[11pt]{amsart}
\usepackage{tabularx,booktabs,tikz}
\usepackage{caption}
\usepackage{amsmath}
\usepackage{amsfonts}

\usepackage{amscd}
\usepackage{amsthm}
\usepackage{amssymb} 
\usepackage{latexsym}
\usepackage{eufrak}
\usepackage{euscript}
\usepackage{epsfig}
\usepackage{graphics}
\usepackage{array}
\usepackage{enumerate}
\usepackage{dsfont}
\usepackage{color}
\usepackage{wasysym}
\usepackage{hyperref}
\usepackage{pdfsync}

\newcommand{\Hmm}[1]{\leavevmode{\marginpar{\tiny%
$\hbox to 0mm{\hspace*{-0.5mm}$\leftarrow$\hss}%
\vcenter{\vrule depth 0.1mm height 0.1mm width \the\marginparwidth}%
\hbox to
0mm{\hss$\rightarrow$\hspace*{-0.5mm}}$\\\relax\raggedright #1}}}

\newtheorem{theorem}{Theorem}[section]
\newtheorem{lemma}[theorem]{Lemma}
\newtheorem{corollary}[theorem]{Corollary}

\usepackage{natbib}
\begin{document}

\bibliographystyle{abbrv}

\title[Existence of normalized solutions to NLS  on lattice graphs]{Existence of normalized solutions to nonlinear Schr\"{o}dinger equations with potential on lattice graphs}

\author{Weiqi Guan}
\address{Weiqi Guan: School of Mathematical Sciences, Fudan University, Shanghai 200433, China}
\email{wqguan24@fudan.edu.cn}

\begin{abstract}
	We study the existence of ground state normalized solution of the following Schr\"{o}dinger equation:
    \begin{equation*}
\left\{
    \begin{aligned}
        &-\Delta u+V(x)u+\lambda u=f(x,u), x\in\mathbb{Z}^d\\
        &\Vert u\Vert_2^2=a
    \end{aligned}
    \right.    
\end{equation*}
    where $V(x)$ is trapping potential or well potential, $f(x,u)$ satisfies Berestycki-Lions type condition and other suitable conditions. We show that there always exists a threshold $\alpha\in[0,\infty)$ such that there do not exist ground state normalized solutions for $a\in (0,\alpha)$, and there exists a ground state normalized solution for $a\in(\alpha,\infty)$. Furthermore, we prove sufficient conditions for the positivity of $\alpha$ that $\alpha=0$ if $f(x,u)$ is mass-subcritical near 0, and $\alpha\textgreater0$ if $f(x,u)$ is mass-critical or mass-supercritical near 0.

\end{abstract}
\par
\maketitle

\bigskip

\section{introduction}
In the present paper we will study existence of standing wave solution to the discrete nonlinear Schr{\"o}dinger equation
\begin{equation}
-i\frac{\partial\psi}{\partial t}-\Delta\psi+V(x)\psi=g(x,\vert\psi\vert)\psi,\quad x\in\mathbb{Z}^d    
\end{equation}
where
\[
\Delta\psi(t,x)=\sum_{y\sim x}(\psi(t,y)-\psi(t,x))
\]
is the discrete Laplacian on the lattice. Standing wave solution refers to a solution of the form
\begin{equation}
\psi(t,x)=e^{-i\lambda t}u(x)    
\end{equation}
Substituting (2) into equation (1), we get 
\begin{equation}
-\Delta u+V(x)u+\lambda u=f(x,u)    
\end{equation}
where $f(x,u)=g(x,\vert u\vert)u$. 
Regarding equation (3), two distinct approaches exist in the literature. The majority of existing studies adopt the frequency prescription approach, where $\lambda$ is fixed a priori (see, e.g., \citep{MR2191617,MR2652479,MR3115848,CHEN20163493,JIA2017568,MR4299009,MR4401801,MR4568177} and references therein). Alternatively, another significant body of work considers the mass-prescription approach, motivated by the fact that standing waves naturally preserve the $L^2$-norm over time. This leads to the problem
\begin{equation}
\left\{
    \begin{aligned}
        &-\Delta u+V(x)u+\lambda u=f(x,u),x\in \mathbb{Z}^d\\
        &\Vert u\Vert_2^2=a
    \end{aligned}
    \right.    
\end{equation}
 To find a solution to problem (4), it suffices to prove the existence of critical points of the functional 
\begin{equation}
    \Phi(u)= \frac{1}{2}\int_{\mathbb{Z}^{d}} |\nabla u|^{2}+V(x)u^{2} \, dx
-\int_{\mathbb{Z}^{d}}F(x,u) \, dx
\end{equation} 
on the sphere $S_a:=\{u\in \mathcal{H} : \Vert u\Vert_2^2=a\}$, where $F(x,u)=\int_{0}^{u} f(x,t) dt$ and $u$ belongs to $\mathcal{H}:=\{u\in l^2(\mathbb{Z}^d):\vert\int_{\mathbb{Z}^d}V(x)u^2dx\vert\textless\infty\}$, equipped with norm $\Vert u\Vert_{\mathcal{H}}^2=\int_{\mathbb{Z}^d}u^2+(V(x)-\inf V(x)) u^2+\vert\nabla u\vert^2dx$. See Section~\ref{sec:2} for definitions. Then $\lambda=\lambda_a$ appears as a Lagrange multiplier. If $\Phi(u)$ is bounded from below on $S_a$, we set
\[
E_a:=\inf_{u\in S_a}\Phi(u)
\]
If $E_a$ is attained by some $u$, then $u$ is called \emph{ground state normalized solution}.\\

Problem (4) is a discrete version of the Schr\"{o}dinger equation on the Euclidean space
\begin{equation}
\left\{
    \begin{aligned}
        &-\Delta u+V(x)u+\lambda u=f(x,u) ,x \in\mathbb{R}^d \\
        &\Vert u\Vert_2^2=a.
    \end{aligned}
    \right.    
\end{equation}
Problem (6) is studied extensively; see, e.g., \citep{MR3147450,MR4096725,MR4107073,MR4390628,MR4803490} and references therein. If $V(x)=0$ and $f(x,u)=\vert u\vert^{p-2}u$ with $2\textless p \textless 2+\frac{4}{d}$, then standard concentration compactness principle \citep{MR778970,MR778974} applies and for all $a\textgreater0$, there always exists ground state normalized solution. If $V(x)=0$ and $f(x,u)=f(u)$ satisfies both Berestycki-Lions type condition and mass-subcritical growth condition, M. Shibata \cite{MR3147450} showed that there exists a threshold for the parameter $a$. Z. Yang et al. \cite{MR4390628} extended M. Shibata's results to a general potential $V(x)$ which is a well potential or a trapping potential.\\
 
Equations on graphs have drawn great attention in the last decade. A. Grigor’yan et al. \cite{MR3523107} first proved the existence result for the Kazdan-Warner equation on a finite graph. For the Schr\"{o}dinger equation on a weighted graph $G=(V,E,\omega,\mu),$ the following is a special case of problem (4) 
\begin{equation}
\left\{
    \begin{aligned}
        &-\Delta u+\lambda u=\vert u\vert^{p-2}u,x\in G\\
        &\Vert u\Vert_2^2=a.
    \end{aligned}
    \right.    
\end{equation} Y. Yang et al. \cite{YANG2024128173} proved that on a finite weighted graph there always exists a ground state normalized solution of  (7) for any $a\textgreater0$ and $p\textgreater2.$ In the setting of lattice graph, M. I. Weinstein \cite{MR1690199} proved that there exists a threshold $\alpha$ to distinguish the existence of ground state normalized solutions, and if $2\textless p\textless 2+\frac{4}{d}$, then $\alpha=0;$ if $p\geq2+\frac{4}{d}$, then $\alpha\textgreater0$. The proof in \cite{MR1690199} relies on the discrete concentration compactness principle. Afterwards, A. G. Stefanov et al. \cite{MR4608774}  gave a detailed proof for the case $ 2\textless p\textless 2+ \frac{4}{d}$ based on Weinstein's work. Also in this case, H. Hajaiej et al. \cite{hajaiej2022discreteschwarzrearrangementlattice} used the method of discrete Schwarz rearrangement to prove that $\alpha=0$. For the Schr\"{o}dinger equation with general nonlinearity and potential, on one hand, the discrete concentration compactness principle doesn't apply since the strict inequality $E_{a+b}\textless E_a+E_b$ may not hold for general $a$ and $b,$ but we have the inequality $E_{a+b}\textless E_a+E_b$ when $E_a$ or $E_b$ is attained. On the other hand, for a general potential $V(x)$ without symmetry, the method of discrete Schwarz rearrangement in \cite{hajaiej2022discreteschwarzrearrangementlattice} doesn't apply. Using methods developed in \citep{MR3147450,MR4390628}, we can prove the main result for the existence of a threshold $\alpha\in [0,\infty)$. 
\begin{theorem}\label{thm1}
     Consider functional (5) on the sphere $S_a.$ Suppose that the nonlinearity $f(x,u)$ satisfies following assumptions:\\
$(f0)$$f(x,\cdot):\mathbb{R}\rightarrow\mathbb{R}$ is continuous for fixed $x\in\mathbb{Z}^d.$\\
$(f1)$ $\lim\limits_{s\rightarrow0}\frac{f(x,s)}{s}=0$ and $\lim\limits_{\vert s\vert\rightarrow\infty}\frac{f(x,s)}{\vert s\vert^q}=0$ uniformly for $x\in \mathbb{Z}^d$ and for some $q\textgreater1$.\\
$(f2)$$\lim\limits_{\vert x\vert\rightarrow\infty}f(x,s)=\tilde{f}(s)$ uniformly for bounded $s$, and $f(x,s)\geq \tilde{f}(s)$. If $f(x,s)\equiv \tilde{f}(s)$ do not hold, we assume there exists $x_1\in\mathbb{Z}^d$ such that $f(x_1,s)\textgreater\tilde{f}(s),$ $s\in\mathbb{R}.$\\
$(f3)$ There exists $\xi\textgreater0$ such that $\tilde{F}(\xi):=\int_{0}^{\xi}\tilde{f}(s)ds\textgreater0$.\\
$(f4)$ $\forall \theta\textgreater1$,$F(x,\sqrt{\theta}s)\textgreater\theta F(x,s)$, $\forall s\neq0,x\in\mathbb{Z}^d$.\\
Let the potential $V(x)$ satisfies\\
$(V_0)$ $V(x)\leq\lim\limits_{\vert x\vert\rightarrow\infty}V(x):=V_{\infty}\in(-\infty,\infty].$\\
Then for all $a\textgreater0$, we have $E_a\textgreater-\infty$, and there exists a threshold $\alpha\in [0,\infty)$ such that if $a\in(0,\alpha)$, $E_a$ can not be attained, if $a\in(\alpha,\infty)$, then $E_a$ is attained by a ground state normalized solution.
\end{theorem}
We note that our assumptions for $f(x,u)$ are much more general than the assumptions in \citep{MR4390628} since we do not require $q=1+\frac{4}{d}$ in condition $(f1)$. Throughout this paper, we always assume $(f0)-(f4)$ conditions for $f(x,u)$. We also give a criterion on whether $\alpha\textgreater0$ when $V_{\infty}\textless\infty$.
\begin{theorem}\label{thm2}
Assume $V(x)$ satisfies condition $(V_0)$ with $V_{\infty}\textless\infty$, then the following hold:\\
(i) 
    \[
\alpha\textless\xi^2([\frac{d\xi^2}{\tilde{F}(\xi)}]+1)^{d}.
    \]
    (ii)If $\varliminf\limits_{ s\rightarrow0}\frac{\tilde{F}(s)}{\vert s\vert^{2+\frac{4}{d}}}=\infty,$ then $\alpha=0$.\\ 
    (iii)For $d\geq 3$, if $\varlimsup\limits_{s\rightarrow0}\frac{F(x,s)}{\vert s\vert^{2+\frac{4}{d}}}\textless\infty$ holds uniformly for $x\in\mathbb{Z}^d$, and there exists some $0\textless\epsilon\textless1$
    \[
    V(x)\geq -C_d(1-\epsilon)\frac{1}{1+\vert x\vert^2},\quad x\in\mathbb{Z}^d
    \]
    where $C_d$ is the best constant in Hardy inequality on $\mathbb{Z}^d$, then $\alpha\textgreater0$. Furthermore, we have 
    \[
    \alpha\geq\min((\frac{\epsilon}{2C_{F}C_{d,2+\frac{4}{d}}})^{\frac{d}{2}},\delta^2)
    \]
    where constant $C_F$ such that $F(x,s)\leq C_F\vert s\vert^{2+\frac{4}{d}}$ holds uniformly for $x\in\mathbb{Z}^d$ and for $\vert s\vert\leq\delta$, and constant $C_{d,2+\frac{4}{d}}$ is the best constant for 
    \[
    \Vert u\Vert_{2+\frac{4}{d}}^{2+\frac{4}{d}}\leq C_{d,2+\frac{4}{d}}\Vert\nabla u\Vert_2^2\Vert u\Vert_2^{\frac{4}{d}}.
    \]
\end{theorem}
The paper is organized as follows:
In next section, we recall the basic setting for analysis on graphs, and some useful results.
In Section~\ref{sec:3}, we prove one of main results, Theorem~\ref{thm1}, for the case of trapping potential $V_\infty=\infty.$
In Section~\ref{sec:4}, we prove Theorem~\ref{thm1} for the case of $V_\infty<\infty,$ and the other main result, Theorem~\ref{thm2}.

\section{Preliminaries}\label{sec:2}
A lattice graph $\mathbb{Z}^d$ is composed of the set of vertices 
\[
\{x=(x_1,..,x_d):x_i\in\mathbb{Z},1\leq i\leq d\}
\]
and the set of edges 
\[E=\{(u,v):u,v\in\mathbb{Z}^d,\sum\limits_{1\leq i\leq d}\vert u_i-v_i\vert=1\}.
\] Two vertices are neighbours if $(u,v)\in E,$ denoted by $u\sim v.$
The ball $B_R(x)$ denotes the set $\{y\in \mathbb{Z}^d:\vert y-x\vert\textless R\}\}$, where $\vert y-x\vert=\sum\limits_{1\leq i\leq d}\vert y_i-x_i\vert$. The boundary of the ball $\partial B_R(x)$ denotes the set $\{y\in\mathbb{Z}^d:y\in B_R^c(x),\exists z\in B_R(x),y\sim z\}.$ We abbreviate $B_R(0)$ as $B_R$ for simplicity. We denote by $C(\mathbb{Z}^d)$ the space of functions on the lattice graph, and by $C_0(\mathbb{Z}^d)$ the space of functions with finite support. Furthermore, we define the $l^p$ summable function space on the lattice as 
\[
l^p(\mathbb{Z}^d):=\{u\in C(\mathbb{Z}^d):\Vert u\Vert_p\textless\infty\}
\]
where $p\in [1,\infty]$ and $l^p$ norm of a function $u\in C(\mathbb{Z}^d)$ is defined as 
\begin{equation*}
\Vert u\Vert_p:=\left\{
\begin{aligned}
    &(\sum_{x\in\mathbb{Z}^d}\vert u(x)\vert^p)^{\frac{1}{p}},1\leq p\textless\infty\\
    &\sup_{x\in\mathbb{Z}^d}\vert u(x)\vert,p=\infty.
\end{aligned}
\right.
\end{equation*}
We write the weak convergence in $l^p(\mathbb{Z}^d)$ as $u_n\rightharpoonup u.$ For $u_n\rightharpoonup u$ in $l^p(\mathbb{Z}^d)$, since the delta function
\begin{equation*}
\delta_x(y)=\left\{
    \begin{aligned}
        &1,y=x \\
        &0, \mathrm{otherwise}
    \end{aligned}
    \right.    
\end{equation*}
belongs to $l^{\frac{p}{p-1}}(\mathbb{Z}^d)$, we have the pointwise convergence $u_n(x)\rightarrow u(x)$ for all $x\in\mathbb{Z}^d$. When $u\in l^1(\mathbb{Z}^d)$, we define the integration of $u$ over $\mathbb{Z}^d$ with respect to the counting measure as 
\[
\int_{\mathbb{Z}^d}u (x)dx :=\sum_{x\in\mathbb{Z}^d}u(x).
\]
For any function $u,v\in C(\mathbb{Z}^d)$, we define the gradient form $\Gamma(u,v)$ as 
\[
\Gamma(u,v)(x)=\frac{1}{2}\sum_{y\sim x}(u(y)-u(x))(v(y)-v(x)).
\]
For simplicty, we write $\Gamma(u):=\Gamma(u,u)$ and
\[
\vert\nabla u\vert(x):=\sqrt{\Gamma(u)}
\]
In the remaining part of this section, we give some preliminary results that will be used throughout this paper. The following lemma is well-known on the lattice; see \citep{MR4328634}.
\begin{lemma}\label{normcontrol}
For $u\in l^{p}(\mathbb{Z}^d)$, we have $u\in l^{q}(\mathbb{Z}^d)$ for all $q\textgreater p$ and $\Vert u\Vert_{q}\leq\Vert u\Vert_p$.
\end{lemma}
Using the above lemma, we can prove the following result, which ensures the functional is bounded from below on the sphere $S_a$.
\begin{lemma}\label{boundness}
    Assume $V(x)$ satisfies condition $(V_0)$, then $\Phi(u)$ is bounded from below on $S_a$ for all $a\textgreater0$. Moreover, the minimizing sequence of $E_a$ is bounded in $\mathcal{H}$. 
\end{lemma}
\begin{proof}
    By conditions $(f0)(f1)(f2)$,  one easily checks that for any $\epsilon\textgreater0$, there exists $C_{\epsilon}$ such that
    \[
    \vert F(x,u)\vert\leq \epsilon\vert u\vert^2+C_{\epsilon}\vert u\vert^{q+1}.
    \]
    From condition $(V_{0}),$  we have $V(x)\geq -C$ for some $C$ and all $x\in \mathbb{Z}^d$. By Lemma \ref{normcontrol} we have an estimate of $\Phi(u)$ over $S_a$ 
    \begin{align*}
        &\Phi(u)=\frac{1}{2}\int_{\mathbb{Z}^d}\vert \nabla u\vert^2dx+\frac{1}{2}\int_{\mathbb{Z}^d}V(x)u^2dx-\int_{\mathbb{Z}^d}F(x,u)dx\\
        &\geq -\frac{C}{2}a-\epsilon\int_{\mathbb{Z}^d}\vert u\vert^2dx-C_{\epsilon}\int_{\mathbb{Z}^d}\vert u\vert^{q+1 }dx\\
        &\geq  -\frac{C}{2}a-\epsilon a-C_{\epsilon} a^{\frac{q+1}{2}}.
    \end{align*}
     
    Given a minimizing sequence $\{u_n\}\subset S_a$ such that $\lim\limits_{n\rightarrow\infty}\Phi(u_n)=E_a$, there is a universal constant $C$ such that $\Phi(u_n)\leq C$, since
    \[
    \Phi(u_n)\geq\frac{1}{2}\Vert u_n\Vert_{\mathcal{H}}^2+(\frac{1}{2}\inf V(x)-\frac{1}{2})a-\epsilon a-C_{\epsilon}a^{\frac{q+1}{2}}
    \]
    Hence, $\{u_n\}$ is a bounded sequence in $\mathcal{H}$.
\end{proof}
Compared to the Euclidean space, by Lemma \ref{boundness} we do not require the nonlinearity to be mass sub-critical to ensure the functional is bounded from below on the sphere on the lattice graph. Next, we recall the discrete Lions lemma corresponding to \citep{MR778970} in $\mathbb{R}^d$, see \citep{MR4564936} for the proof.
\begin{lemma}\label{vanish}
    Let $1\leq r\textless\infty$. Assume that $\{ u_n\}$ is bounded in $l^{r}(\mathbb{Z}^d)$ and $\Vert u_n\Vert_{\infty}\rightarrow 0$ as $n\rightarrow\infty$. Then for any $r\textless s\textless\infty$, as $n\rightarrow\infty$
    \[
    \Vert u_n\Vert_{s}\rightarrow0.
    \]
\end{lemma}
We also need the following lemma.
\begin{lemma}\label{Fvanish}
    Let $\{u_n\}\subset l^2(\mathbb{Z}^d)$ be a bounded sequence. If either 
    $\lim\limits_{n\rightarrow\infty}\Vert u_n\Vert_2=0$ or $\lim\limits_{n\rightarrow\infty}\Vert u_n\Vert_{q+1}=0$, then we have 
    \[
    \lim_{n\rightarrow\infty}\int_{\mathbb{Z}^d}F(x,u_n)dx=0.
    \]
\end{lemma}
\begin{proof}
    For any $\epsilon\textgreater0$, there exists $C_{\epsilon}\textgreater0$ such that
    \[
    \vert F(x,u)\vert\leq \epsilon \vert u\vert^2+C_{\epsilon}\vert u\vert^{q+1}.
    \]
    If $\lim\limits_{n\rightarrow\infty}\Vert u_n\Vert_2=0$, then we have 
    \begin{align*}
        &\vert\int_{\mathbb{Z}^d}F(x,u_n)dx\vert\\
        &\leq \int_{\mathbb{Z}^d}\epsilon\vert u_n\vert^2+C_{\epsilon}\vert u_n\vert^{q+1}dx\\
        &\leq \int_{\mathbb{Z}^d}\epsilon\vert u_n\vert^2dx+C_{\epsilon}(\int_{\mathbb{Z}^d}\vert u_n\vert^2dx)^{\frac{q+1}{2}}.
    \end{align*}
    Hence letting $n\rightarrow\infty$ we have $\lim\limits_{n\rightarrow\infty}\vert\int_{\mathbb{Z}^d}F(x,u_n)dx\vert=0$. If $\lim\limits_{n\rightarrow\infty}\Vert u_n\Vert_{q+1}=0$, then by a similar argument we have
    \begin{align*}
        &\vert\int_{\mathbb{Z}^d}F(x,u_n)dx\vert\\
        &\leq\int_{\mathbb{Z}^d}\epsilon\vert u_n\vert^2+C_{\epsilon}\vert u_n\vert^{q+1}dx\\
        &\leq C\epsilon+C_{\epsilon}\int_{\mathbb{Z}^d}\vert u_n\vert^{q+1}dx.
    \end{align*}
    Letting $n\rightarrow\infty$, we have
    \[
    \varlimsup\limits_{n\rightarrow\infty}\vert\int_{\mathbb{Z}^d}F(x,u_n)dx\vert\leq a\epsilon.
    \]
    Since $\epsilon$ is arbitrary, we have $\lim\limits_{n\rightarrow\infty}\vert\int_{\mathbb{Z}^d}F(x,u_n)dx\vert=0.$
\end{proof}
We introduce the classical Brezis-Lieb lemma \citep{MR699419}.
\begin{lemma}\label{BreLieb}
    Let $(\Omega,\Sigma,\tau)$ be a measure space, where $\Omega$ is a set equipped with a $\sigma-$algebra $\Sigma$ and a Borel measure $\tau:\Sigma\rightarrow[0,\infty]$. Given a sequence $\{u_n\}\subset L^p(\Omega,\Sigma,\tau)$ with $0\textless p\textless\infty$. If $\{u_n\}$ is uniformly bounded in $L^p(\Omega,\Sigma,\tau)$ and $u_n\rightarrow u$, $\tau$-a.e. in $\Omega$, then we have that
    \[
    \lim\limits_{n\rightarrow\infty}(\Vert u_n\Vert_{L^p(\Omega)}^p-\Vert u_n-u\Vert_{L^p(\Omega)}^p)=\Vert u\Vert_{L^p(\Omega)}^p.
    \]
\end{lemma}
A direct consequence of Lemma \ref{BreLieb} is the following; see also \citep{MR4564936}.
\begin{corollary}\label{NormBreLieb}
For a bounded sequence $\{u_n\}\subset l^2(\mathbb{Z}^d)$ such that $u_n\rightharpoonup u$ in $l^2(\mathbb{Z}^d)$, we have\\
    (i)\[
    \lim\limits_{n\rightarrow\infty}(\Vert u_n\Vert_{l^2(\mathbb{Z}^d)}^2-\Vert u_n-u\Vert_{l^2(\mathbb{Z}^d)}^2)=\Vert u\Vert_{l^2(\mathbb{Z}^d)}^2,
    \]
    (ii)
    \[
    \lim\limits_{n\rightarrow\infty}(\Vert \nabla u_n\Vert_{l^2(\mathbb{Z}^d)}^2-\Vert \nabla (u_n- u)\Vert_{l^2(\mathbb{Z}^d)}^2)=\Vert \nabla u\Vert_{l^2(\mathbb{Z}^d)}^2.
    \]
\end{corollary}
We also need a Brezis-Lieb type lemma for $F(x,u)$ as follows, whose proof is similar to \citep[Lemma 2.2]{MR1043058}.
\begin{lemma}\label{FBreLieb}
    For a bounded sequence $\{u_n\}\subset l^2(\mathbb{Z}^d)$ such that $u_n\rightharpoonup u$ in $l^2(\mathbb{Z}^d)$, we have
    \[
    \lim_{n\rightarrow\infty}\int_{\mathbb{Z}^d}F(x,u_n)-F(x,u)-F(x,u_n-u)dx=0
    \]
\end{lemma}
\begin{proof}
    We directly compute that
    \[
    \begin{aligned}
        &\vert\int_{\mathbb{Z}^d}F(x,u_n)-F(x,u)-F(x,u_n-u)dx\vert\\
        &\leq\vert\int_{B_R}(F(x,u_n)-F(x,u))dx\vert+\vert\int_{B_R}F(x,u_n-u)dx\vert+\vert\int_{\mathbb{Z}^d\backslash B_R}F(x,u)dx\vert\\
        &\quad+\vert\int_{\mathbb{Z}^d\backslash B_R}(F(x,u_n)-F(x,u_n-u))dx\vert.
    \end{aligned}
    \]
    For the last term, by H\"{o}lder's inequality and the boundness of $u_n$, we have
    \[
    \begin{aligned}
        &\vert\int_{\mathbb{Z}^d\backslash B_R}(F(x,u_n)-F(x,u_n-u))dx\vert\\
        &=\vert\int_{\mathbb{Z}^d\backslash B_R}(F(x,u+(u_n-u))-F(x,u_n-u))dx\vert\\
        &=\vert\int_{\mathbb{Z}^d\backslash B_R}f(x,\theta u+(u_n-u))udx\vert\\
        &\leq C\int_{\mathbb{Z}^d\backslash B_R}\vert \theta u+(u_n-u)\vert\vert u\vert +\vert\theta u+(u_n-u)\vert^{q}\vert u\vert dx\\
        &\leq C((\int_{\mathbb{Z}^d\backslash B_R}\vert u\vert^2dx)^{\frac{1}{2}}+(\int_{\mathbb{Z}^d\backslash B_R}\vert u\vert^{q+1}dx)^{\frac{1}{q+1}})\\
        &+C(\int_{\mathbb{Z}^d\backslash B_R}\vert u\vert^2dx+\int_{\mathbb{Z}^d\backslash B_R}\vert u\vert^{q+1}dx).
        \end{aligned}
    \]
    Therefore, for any $\epsilon\textgreater0$, we can take $R$ sufficiently large such that 
    \[
    \begin{aligned}
        \vert\int_{\mathbb{Z}^d\backslash B_R}(F(x,u_n)-F(x,u_n-u))dx\vert\leq\epsilon
    \end{aligned}
    \]
    and 
    \[
    \begin{aligned}
        &\vert\int_{\mathbb{Z}^d\backslash B_R}F(x,u)dx\vert\\
        &\leq C(\int_{\mathbb{Z}^d\backslash B_R}\vert u\vert^2dx+\int_{\mathbb{Z}^d\backslash B_R}\vert u\vert^{q+1}dx)\\
        &\leq\epsilon.
    \end{aligned}
    \]
    Now fix $R.$ For the first two terms, since $u_n\rightharpoonup u$ in $l^2(\mathbb{Z}^d)$, for any $ x\in\mathbb{Z}^d$, $u_n(x)\rightarrow u(x)$. Thus we have 
    \[
    \begin{aligned}
        &\lim_{n\rightarrow\infty}\vert\int_{B_R}(F(x,u_n)-F(x,u))dx\vert=0,\\
        &\lim_{n\rightarrow\infty}\vert\int_{B_R}F(x,u_n-u)dx\vert=0.
    \end{aligned}
    \]
    Hence
    \[
    \varlimsup\limits_{n\rightarrow\infty}\vert\int_{\mathbb{Z}^d}F(x,u_n)-F(x,u)-F(x,u_n-u)dx\vert\leq C\epsilon.
    \]
    Since $\epsilon$ is arbitrary, we get the desired results.
\end{proof}
Combining Corollary \ref{NormBreLieb} with Lemma \ref{FBreLieb}, we have the following lemma.
\begin{lemma}\label{PhiBreLieb}
     For a bounded sequence $\{u_n\}\subset l^2(\mathbb{Z}^d)$ such that $u_n\rightharpoonup u$ in $l^2(\mathbb{Z}^d)$, suppose that V(x) satisfies condition $(V_0)$ with $V_{\infty}=0$, then we have
    \[
    \lim_{n\rightarrow\infty}(\Phi(u_n)-\Phi(u_n-u))=\Phi(u).
    \]
\end{lemma}
\begin{proof}
    By Corollary \ref{NormBreLieb} and Lemma \ref{FBreLieb}, it suffices to prove that 
    \[
    \lim\limits_{n\rightarrow\infty}(\int_{\mathbb{Z}^d}V(x)u_n^2dx-\int_{\mathbb{Z}^d}V(x)(u_n-u)^2dx)=\int_{\mathbb{Z}^d}V(x)u^2dx.
    \]
    For any $\epsilon\textgreater0$, there exists sufficently large $R$ such that $\vert V(x)\vert\leq \epsilon$ for all $x\in B_{R}^c.$ Hence, we have
    \begin{align*}
         &\vert\int_{\mathbb{Z}^d}V(x)u_n^2dx-\int_{\mathbb{Z}^d}V(x)(u_n-u)^2dx-\int_{\mathbb{Z}^d}V(x)u^2dx\vert\\
         &\leq\vert\int_{B_R}V(x)u_n^2dx-\int_{B_R}V(x)(u_n-u)^2dx-\int_{B_R}V(x)u^2dx\vert\\
         &+\vert\int_{B_R^c}V(x)u_n^2dx-\int_{B_R^c}V(x)(u_n-u)^2dx-\int_{B_R^c}V(x)u^2dx\vert\\
         &\leq\vert\int_{B_R}V(x)u_n^2dx-\int_{B_R}V(x)(u_n-u)^2dx-\int_{B_R}V(x)u^2dx\vert+C\epsilon.
    \end{align*}
    Since $u_n\rightharpoonup u$, for all $x\in B_R$, we have $u_n(x)\rightarrow u(x)$ as $n\rightarrow\infty$. 
    Hence, we have
    \[\varlimsup\limits_{n\rightarrow\infty}\vert(\int_{\mathbb{Z}^d}V(x)u_n^2dx-\int_{\mathbb{Z}^d}V(x)(u_n-u)^2dx)-\int_{\mathbb{Z}^d}V(x)u^2dx\vert\leq C\epsilon.
    \]
    Since $\epsilon$ is arbitrary, we prove the desired result.
\end{proof}

\section{Trapping potential case $V_{\infty}=\infty$}\label{sec:3}
In this section, we prove Theorem~\ref{thm1} for the case $V_{\infty}=\infty$, i.e. so-called trapping potential. Similar to the Euclidean space, trapping potential on the lattice can provide enough compactness, and we have the following well-known lemma; see also \citep{MR2671878}.
\begin{lemma}\label{lem:comp}
    Suppose that V(x) satisfies condition $(V_0)$ with $V_{\infty}=\infty$, then $\mathcal{H}$ compactly embeds into $l^{p}(\mathbb{Z}^d)$ for $p\in[2,\infty]$.
\end{lemma}
With the help of the above lemma, we can derive the existence of a ground state normalized solution in the trapping potential case.
\begin{theorem}\label{maintrap}
    Suppose $V_{\infty}=\infty$, then for all $a\textgreater0$, $E_a$ can be attained by a ground state normalized solution $u$.
\end{theorem}
\begin{proof}
    Take a minimizing sequence $\{u_n\}\subset S_a.$ By Lemma \ref{boundness}, $\{u_n\}$ is a bounded sequence in $\mathcal{H}$. Taking a subsequence if necessary we assume $u_n\rightharpoonup u$ in $\mathcal{H}$, hence $u_n\rightarrow u$ in $l^{p}(\mathbb{Z}^d)$ for $p\in[2,\infty]$ by Lemma \ref{lem:comp}. Combining this with Lemma \ref{Fvanish}, we have $$\lim\limits_{n\rightarrow\infty}\int_{\mathbb{Z}^d}F(x,u-u_n)d\mu=0.$$ Hence by Lemma \ref{FBreLieb} we obtain that $$\lim\limits_{n\rightarrow\infty} \int_{\mathbb{Z}^d}F(x,u_n)d\mu=\int_{\mathbb{Z}^d}F(x,u)d\mu.$$ Also since $$\Vert\nabla (u_n-u)\Vert_2^2\leq C\Vert u_n-u\Vert_2^2$$ and $\lim\limits_{n\rightarrow\infty}\Vert u_n-u\Vert_2^2=0$, by Corollary \ref{NormBreLieb}, we have that $\Vert u\Vert_2^2=a$ and $$\lim\limits_{n\rightarrow\infty}\int_{\mathbb{Z}^d}\vert\nabla u_n\vert^2dx=\int_{\mathbb{Z}^d}\vert\nabla u\vert^2dx.$$ By weak lower semi-continuity of the norm $\Vert\cdot\Vert_{\mathcal{H}}$, we have $$E_a\leq \Phi(u)\leq \lim\limits_{n\rightarrow\infty}\Phi(u_n)=E_a,$$ therefore $E_a$ is attained by $u$.
\end{proof}
\section{Well potential case: $V_{\infty}\textless\infty$}\label{sec:4}
In this section, we are concerned with $V_{\infty}\textless\infty,$ and prove main results. Without loss of generality, we assume $V_{\infty}=0,$ otherwise we replace $(V(x),\lambda)$ by $(V(x)-V_{\infty},\lambda+V_{\infty})$. Since $V(x)$ is bounded, the norm $\Vert u\Vert_{\mathcal{H}}$ is equivalent to $\Vert u\Vert_2$ and $\mathcal{H}$ coincides with $l^2(\mathbb{Z}^d)$. First, we need the following helpful lemma.
\begin{lemma}\label{Vvanish}
    Let $\{u_n\}$ be a bounded sequence in $l^2(\mathbb{Z}^d)$ such that
 for all $x\in\mathbb{Z}^d$, $\lim\limits_{n\rightarrow\infty}u_n(x)=0$. Suppose $V(x)$ satisfies condition $(V_0)$ with $V_{\infty}=0$, then we have
    \[
    \lim\limits_{n\rightarrow\infty}\int_{\mathbb{Z}^d}V(x)u_n^2dx=0.
    \]
\end{lemma}
\begin{proof}
    Let $\epsilon\textgreater0.$ For sufficiently large $R$  such that $\vert V(x)\vert\leq \epsilon$ for all $x\in B_{R}^c$,  we have
    \begin{align*}
    \vert\int_{\mathbb{Z}^d}V(x)u_n^2dx\vert&\leq \vert\int_{B_R}V(x)u_n^2dx\vert+\vert\int_{B_R^c}V(x)u_n^2dx\vert \\
    &\leq \vert\int_{B_R}V(x)u_n^2dx\vert+C\epsilon.
    \end{align*}
    Since for all $x\in B_R$, we have $u_n(x)\rightarrow0$. Hence letting $n\rightarrow\infty$ we have
    \[
    \varlimsup\limits_{n\rightarrow\infty}\vert\int_{\mathbb{Z}^d}V(x)u_n^2dx\vert\leq C\epsilon.
    \]
    Since $\epsilon$ is arbitrary, we prove the desired result.
\end{proof}
To find the ground state normalized solution, we need to establish various properties for the function $a\mapsto E_a$.
\begin{lemma}\label{funproper}
     (1) For all $a\textgreater0$ we have $E_a\leq0$, and there exists $a\textgreater 0$ such that $E_a\textless0$.\\
     (2) $E_{\theta a}\leq\theta E_{a}$ for any $\theta\textgreater1$.\\
     (3) If further $E_a$ is attained, then $E_{\theta a}\textless\theta E_{a}$ for $\theta\textgreater1$.\\
     (4) $E_{a+b}\leq E_a+E_b$ for $a,b\textgreater0$.\\
     (5) If $E_a$ or $E_b$ is attained, then $E_{a+b}\textless E_a+E_b$.\\
     (6) The function $a\mapsto E_a$ is nonincreasing and continuous for $a\textgreater0$.
\end{lemma}
\begin{proof}
    (1) Let
    
    \[
    u_{n}(x)=\left\{
    \begin{aligned}
       & c_{a,n,d}\frac{n-\vert x\vert}{n^{\frac{d}{2}+1}}, \vert x\vert\leq n\\
       & 0, \mathrm{otherwise}
    \end{aligned}
    \right.
    \]
    where $c_{a,n,d}$ is a normalized constant such that $\Vert u_n\Vert_2^2=a.$ Then we have $\Vert\nabla u_n\Vert_2^2\leq Cn^{-2}$ and $\Vert u_n\Vert_{q+1}^{q+1}=o(1)$; see also  \citep[Lemma 3]{MR4608774}. Hence by Lemma \ref{Fvanish} we have $\lim\limits_{n\rightarrow\infty}\int_{\mathbb{Z}^d} F(x,u_n)dx=0$. Also, since $$\lim\limits_{n\rightarrow\infty}\Vert u_n\Vert_{\infty}\leq\lim\limits_{n\rightarrow\infty}\Vert u_n\Vert_{q+1}=0,$$ by Lemma \ref{Vvanish} we have $\lim\limits_{n\rightarrow\infty}\int_{\mathbb{Z}^d}V(x)u_n^2dx=0$. Hence, $\lim\limits_{n\rightarrow\infty}\Phi(u_n)=0$, which implies $E_a\leq0$. Moreover, we consider the function 
    \[
    u_{R}(x)=\left\{
    \begin{aligned}
       & \xi, \vert x\vert_{\infty}\leq R\\
       & 0, \mathrm{otherwise}
    \end{aligned}
    \right.
    \] 
    where $\vert x\vert_{\infty}=\max\limits_{i}\vert x_i\vert$ for $x=(x_1,...,x_d)$ and $R \in\mathbb{N}$. Then 
    \begin{align*}
    \Phi(u_R)&=\frac{1}{2}\int_{\mathbb{Z}^d}\vert\nabla u_R\vert^2+V(x)u_R^2dx-\int_{\mathbb{Z}^d}F(x,u_R)dx\\
        &\leq\frac{1}{2}\int_{\mathbb{Z}^d}\vert\nabla u_R\vert^2dx-\int_{\mathbb{Z}^d}\tilde{F}(u_R)dx\\
        &\leq d\xi^2 (2R+1)^{d-1}-\tilde{F}(\xi)(2R+1)^d.
    \end{align*}
    Since $\tilde{F}(\xi)\textgreater0$, for sufficiently large $R$ we have $\Phi(u_R)\textless0$ and hence for $a = \Vert u_R\Vert_2^2$, we have $E_a\textless 0$. \\
    
    (2) By condition $(f4)$, for any $\theta\textgreater1$ and $u\in l^2(\mathbb{Z}^d)$ such that $\Vert u\Vert_2^2=a,\Phi(u)\leq E_{a}+\epsilon$, we directly compute that
    \[
    \begin{aligned}
    \Phi(\sqrt{\theta} u)&=\frac{\theta}{2}\int_{\mathbb{Z}^{d}} |\nabla u|^{2}+V(x)u^{2} dx
-\int_{\mathbb{Z}^{d}}F(x,\sqrt{\theta} u) dx\\
&\textless \frac{\theta}{2}\int_{\mathbb{Z}^{d}} |\nabla u|^{2}+V(x)u^{2} dx
-\theta\int_{\mathbb{Z}^{d}}F(x,u) dx\\
&=\theta\Phi(u)\\
&\leq \theta E_a+\theta\epsilon.
\end{aligned}
    \]
    Since $\Phi(\sqrt{\theta} u )\geq E_{\theta a}$, letting $\epsilon\rightarrow0$ we have 
    \[
    E_{\theta a}\leq \theta E_{a}.
    \]
    
    (3) If $E_a$ is attained, then in the proof of (2), we take $\Phi(u)=E_a$ directly and get $E_{\theta a}\textless \theta  E_{a}$.\\
    
    (4) Without loss of generality, we assume $a\textgreater b$, then by (2) we have
    \[
    E_{a+b}=E_{(1+\frac{b}{a})a}\leq(1+\frac{b}{a})E_a=E_a+\frac{b}{a}E_{\frac{a}{b}b}\leq E_a+E_b.
    \]
    
    (5) This follows directly from (3) and (4).\\
    
    (6) Assume $b\textgreater a\textgreater0$, since $E_a\leq0$ for $a\textgreater0$, from (2) we have
    \[
    E_{b}\leq\frac{b}{a}E_{a}\leq E_a.
    \]
    Therefore function $a\mapsto E_{a}$ is non-increasing. For continuity, suppose $a_n\rightarrow a$, take $u_n\in S_{a_n}$ such that $\Phi(u_n)\leq E_{a_n}+\frac{1}{n}$, then we have
    \[
    \begin{aligned}
        &\vert\int_{\mathbb{Z}^d}F(x,\sqrt{\frac{a}{a_n}}u_n)dx-\int_{\mathbb{Z}^d}F(x,u_n)dx\vert\\
        &=\vert\int_{\mathbb{Z}^d}\int_{0}^1f(x,u_n+(\sqrt{\frac{a}{a_n}}-1)\theta u_n)(\sqrt{\frac{a}{a_n}}-1)u_nd\theta dx\vert\\
        &\leq C\vert\sqrt{\frac{a}{a_n}}-1\vert\int_{\mathbb{Z}^d}\vert u_n\vert^2+\vert u_n\vert^{q+1}dx\\
        &\leq C\vert\sqrt{\frac{a}{a_n}}-1\vert(a+a^{\frac{q+1}{2}}).
    \end{aligned}
    \]
    Therefore
    \[
    \int_{\mathbb{Z}^d}F(x,\sqrt{\frac{a}{a_n}}u_n)dx=\int_{\mathbb{Z}^d}F(x,u_n)dx+o(1).
    \]
    Moreover,
    \begin{align*}
    &\frac{a}{2a_n}\int_{\mathbb{Z}^d}\vert\nabla u_n\vert^2+V(x)u_n^2dx\\
    &=\frac{1}{2}\int_{\mathbb{Z}^d}\vert\nabla u_n\vert^2+V(x)u_n^2dx+\frac{a-a_n}{2a_n}\int_{\mathbb{Z}^d}\vert\nabla u_n\vert^2+V(x)u_n^2dx\\
    &\leq\frac{1}{2}\int_{\mathbb{Z}^d}\vert\nabla u_n\vert^2+V(x)u_n^2dx+ C\vert\frac{a-a_n}{2a_n}\vert\int_{\mathbb{Z}^d}\vert u_n\vert^2dx\\
    &=\frac{1}{2}\int_{\mathbb{Z}^d}\vert\nabla u_n\vert^2+V(x)u_n^2dx+o(1).
    \end{align*}
    Therefore, we have
    \[
    \begin{aligned}
        E_{a}&\leq\Phi(\sqrt{\frac{a}{a_n}}u_n)\\
        &=\frac{a}{2a_n}\int_{\mathbb{Z}^d}\vert\nabla u_n\vert^2+V(x)u_n^2dx-\int_{\mathbb{Z}^d}F(x,\sqrt{\frac{a}{a_n}}u_n)dx\\
        &=\Phi(u_n)+o(1)\\
        &\leq E_{a_n}+\frac{1}{n}+o(1).
    \end{aligned}
    \]
    Similarly taking a minimizing sequence $\{v_n\}\subset S_a$, we have
    \[
    E_{a_n}\leq\Phi(\sqrt{\frac{a_n}{a}}v_n)=\Phi(v_n)+o(1)=E_a+o(1).
    \]
    In conclusion, $\vert E_a-E_{a_n}\vert=o(1)$ and hence the function $a\mapsto E_a$ is continuous.
\end{proof}

Before the proof of Theorem \ref{thm1} for discrete nonlinear Schr{\"o}dinger equation with well potential and non-autonomous nonlinearities, we consider the limit functional
\[
\Phi^{\infty}(u)= \frac{1}{2}\int_{\mathbb{Z}^{d}} |\nabla u|^{2}\, dx
-\int_{\mathbb{Z}^{d}}\tilde{F}(u) \,dx.
\]
Limit functional is a special case of (5) such that $V(x)=0$ and $F(x,u)=\tilde{F}(u)$.  Set $E^{\infty}_a :=\inf\limits_{u\in S_a}\Phi^{\infty}(u)$. Then we try to find a threshold for $E_a^{\infty}$ before proving the main results for $E_a$. For our purposes, we need the following lemma.
\begin{lemma}\label{limitcomp}
   Suppose $\{u_n\}$ is a minimizing sequence of $E^{\infty}_a$ for $\Phi^{\infty}$, then one of the following holds:\\
    (i) $\lim\limits_{n\rightarrow\infty}{\Vert u_n(x)\Vert_{\infty}}=0.$\\
    (ii) There exists a family $\{y_n\}\subset\mathbb{Z}^d$ and a global minimizer $u$ such that if taking a subsequence necessarily, we have $
    u_n(\cdot-y_n)\rightarrow u$ in $l^2(\mathbb{Z}^d).$
\end{lemma}
\begin{proof}
    Suppose (i) is not true, then taking a subsequence if necessary, we have
    \[
    0\textless\lim_{n\rightarrow\infty}\sup_{x\in\mathbb{Z}^d}\vert u_n(x) \vert.
    \]
Suppose $\sup_{x\in\mathbb{Z}^d}\vert u_n(x) \vert=\vert u_n(y) \vert$, then $u_n(\cdot-y_n)$ is bounded in $l^2(\mathbb{Z}^d)$ by Lemma \ref{boundness}, therefore there exists $u\in l^2(\mathbb{Z}^d)$ such that $u_n(\cdot-y_n)\rightharpoonup u$ in $l^2(\mathbb{Z}^d)$. Set $v_n=u_n(\cdot-y_n)-u$, then by Lemma \ref{PhiBreLieb} we have 
    \[
    \Phi^{\infty}(u_n)-\Phi^{\infty}(u)-\Phi^{\infty}(v_n)=o(1).
    \]
    
    We claim that $\lim\limits_{n\rightarrow\infty}\Vert v_n\Vert_{\infty}=0$. Suppose it is not true, then taking a subsequence if necessary, we find $\{z_n\}\subset\mathbb{Z}^d$ and $v\in l^2(\mathbb{Z}^d)$ such that $v_n(\cdot-z_n)\rightharpoonup v$ in $l^2(\mathbb{Z}^d)$. 
    Set $\omega_n= v_n(\cdot-z_n)-v$, then by Lemma \ref{PhiBreLieb} again we have
    \[
    \Phi^{\infty}(v_n)=\Phi^{\infty}(v)+\Phi^{\infty}(\omega_n)+o(1).
    \]
    Therefore we have 
    \[
    \Phi^{\infty}(u_n)=\Phi^{\infty}(u)+\Phi^{\infty}(v)+\Phi^{\infty}(\omega_n)+o(1).
    \]
    Set $\beta = \Vert u\Vert_{2}^2$, $\gamma=\Vert v\Vert_{2}^2$, $\delta = a-\beta-\gamma$, then by Corollary \ref{NormBreLieb}, $\Vert\omega_n\Vert_2^2 = \delta+o(1)$.\\
    
    If $\delta\textgreater0$, then we set $\tilde{\omega}_n=\frac{\sqrt{\delta}}{\Vert \omega_n\Vert_{2}}\omega_n$. Similar to the proof of (6) in Lemma \ref{funproper}, we have $\Phi^{\infty}(\omega_n)=\Phi^{\infty}(\tilde{\omega}_n)+o(1)$. Therefore, by (4) in Lemma \ref{funproper}, we have 
    \[
    \begin{aligned}
    \Phi^{\infty}(u_n)&=\Phi^{\infty}(u)+\Phi^{\infty}(v)+\Phi^{\infty}(\omega_n)+o(1)\\
    &=\Phi^{\infty}(u)+\Phi^{\infty}(v)+\Phi^{\infty}(\tilde{\omega}_n)+o(1)\\
    &\geq \Phi^{\infty}(u)+\Phi^{\infty}(v)+E^{\infty}_{\delta}+o(1)\\
    &\geq E^{\infty}_{\beta}+E^{\infty}_{\gamma}+E^{\infty}_{\delta}+o(1)\\
    &\geq E^{\infty}_{a}+o(1).
    \end{aligned}
    \]
    Taking $n\rightarrow\infty$, we have all inequalities become equalities and $u,v$ are global minimizers of $E_{\beta}^{\infty}, E_{\gamma}^{\infty}$ respectively. But then by (5) in Lemma \ref{funproper}, $E^{\infty}_{\beta}+E^{\infty}_{\gamma}\textgreater E^{\infty}_{\beta+\gamma}$, which is a contradiction.\\
    
    If $\delta =0$, then $\lim\limits_{n\rightarrow\infty}\Vert\omega_n\Vert_{2}=0$, which implies $\lim\limits_{n\rightarrow\infty}\Phi^{\infty}(\omega_n)=0$ by Lemma \ref{Fvanish}. Therefore we have
    \[
    E^{\infty}_{a} = \Phi^{\infty}(u)+\Phi^{\infty}(v)\geq E^{\infty}_{\beta}+E^{\infty}_{\gamma}\geq E^{\infty}_a.
    \]
    Thus, $u,v$ are global minimizers of $E_{\beta}^{\infty}, E_{\gamma}^{\infty}$ respectively, but this is a contradiction to (5) in Lemma \ref{funproper}.\\
    
Hence, we prove the claim.  Therefore by Lemma \ref{vanish} we have $\lim\limits_{n\rightarrow\infty}\Vert v_n\Vert_{q+1}=0$, which implies $\lim\limits_{n\rightarrow\infty}\int_{\mathbb{Z}^d}\tilde{F}(v_n)dx=0$ by Lemma \ref{Fvanish}. Now we prove that $\lim\limits_{n\rightarrow\infty}\Vert v_n\Vert_{2}=0$. Otherwise, without loss of generality, taking a subsequence if necessary, we may assume $\lim\limits_{n\rightarrow\infty}\Vert v_n\Vert_{2}\textgreater0.$ Since 
    \[
    \Phi^{\infty}(u_n)=\Phi^{\infty}(u)+\Phi^{\infty}(v_n)+o(1)
    \]
    and 
    \[
    \varliminf_{n\rightarrow\infty}\Phi^{\infty}(v_n)\geq -\lim_{n\rightarrow\infty}\int_{\mathbb{Z}^d}\tilde{F}(v_n)dx=0,
    \]
    we have 
    \[
    E^{\infty}_{a}\geq \Phi^{\infty}(u)\geq E^{\infty}_{\beta}\geq E^{\infty}_{a}.
    \]
    Thus $u$ is a global minimizer of $E^{\infty}_{\beta}$, but then $E^{\infty}_{\beta}\textgreater E^{\infty}_{a}$, which is a contradiction. Hence, we prove the result $\lim\limits_{n\rightarrow\infty}\Vert v_n\Vert_2=0.$ By Corollary \ref{NormBreLieb}, we have $\Vert u\Vert_2^2=a$. Since $\lim\limits_{n\rightarrow\infty}\int_{\mathbb{Z}^d}\tilde{F}(v_n)dx=0$,  $\lim\limits_{n\rightarrow\infty}\Phi^{\infty}(v_n)=0$. By Lemma \ref{PhiBreLieb} we have $\Phi^{\infty}(u)=\lim\limits_{n\rightarrow\infty}\Phi^{\infty}(u_n)=E^{\infty}_a$.
\end{proof}
 Using Lemma \ref{limitcomp}, we are able to find a ground state normalized solution when $E_a^{\infty}\textless0$.
\begin{theorem}\label{limitmain}
    If $E_a^{\infty}\textless0$, then there exists $u\in S_a$ such that $\Phi^{\infty}(u)=E^{\infty}_a$. Moreover, let $\alpha^{\infty}=inf\{a:E^{\infty}_a\textless 0\}$, then if $a\textgreater\alpha^{\infty}$, $E_a^{\infty}$ is attained, if $0\textless a\textless\alpha^{\infty}$, $E_a^{\infty}$ can not be attained.
\end{theorem}
\begin{proof}
    Taking a minimizing sequence $\{u_n\}$ for $E_a^{\infty}$, we exclude the case $\lim\limits_{n\rightarrow\infty}\Vert u_n\Vert_{{\infty}}=0$. Otherwise, suppose that $\lim\limits_{n\rightarrow\infty}\Vert u_n\Vert_{{\infty}}=0$, by Lemma \ref{vanish} we have $u_n\rightarrow0$ in $l^{q+1}(\mathbb{Z}^d)$, therefore by Lemma \ref{Fvanish} we have $\lim\limits_{n\rightarrow\infty}\int_{\mathbb{Z}^d}\tilde{F}(u_n)dx=0.$ Thus 
    \[
    E_a^{\infty}=\lim\limits_{n\rightarrow\infty}\Phi^{\infty}(u_n)\geq-\lim\limits_{n\rightarrow\infty}\int_{\mathbb{Z}^d}\tilde{F}(u_n)dx=0.
    \]
    This contradicts $E_a^{\infty}\textless 0$. Therefore, by Lemma \ref{limitcomp}, there exists a global minimizer $u$. \\
    
    For the rest of the proof, if $a\textgreater\alpha^{\infty},$ then $E_a^{\infty}\textless0$ and thus $E_a^{\infty}$ is attained. If $0\textless a\textless\alpha^{\infty}$, $E_a^{\infty}$ can not be attained. If not, by (5) of Lemma \ref{funproper} we have $E^{\infty}_{\alpha^{\infty}}\textless E_a^{\infty}\leq 0$, but this contradicts $E_{\alpha^{\infty}}^{\infty}=0$. This proves the result.
\end{proof}

Afterward, we give a criterion for $\alpha^{\infty}=0$, which is closely related to the discrete Gagliardo-Nirenberg-Sobolev inequality.
We first recall a partial result of discrete Gagliardo-Nirenberg-Sobolev inequality, which was proved in \citep{MR4608774}.
\begin{theorem}\label{partialGNS}
     Suppose $d\geq 1$, if $2\textless p\leq p_d$ with
    
    \begin{equation*}
p_d:=\left\{
\begin{aligned}
    &\infty,d=1,2\\
    &\frac{2d}{d-2},d\geq 3
\end{aligned}
\right.
\end{equation*}
    then for every $0\textless\theta\textless\min(1,d(\frac{1}{2}-\frac{1}{p}))$, there exists constant $C_{\theta,p}$ such that the following inequality holds for any function $u\in l^2(\mathbb{Z}^d)$
\begin{equation}
\Vert u\Vert_p\leq C_{\theta,p}\Vert\nabla u\Vert_2^{\theta}\Vert u\Vert_2^{1-\theta} .  
\end{equation}
For $p\textgreater p_d$, (8) holds with $\theta =1$.
\end{theorem}

However Theorem \ref{partialGNS} do not address the case $2\textless p\leq p_d$ and $\theta = \min(1,d(\frac{1}{2}-\frac{1}{p}))$. We give a supplementary proof in the case $d\geq 3$ and $d=1$. First, recall the discrete Sobolev inequality \citep{HUA20156162}; see also \citep{MR4328634}. 
\begin{theorem}\label{sobolev}
    Suppose $d\geq 2$, then for any $u\in D^{1,p}(\mathbb{Z}^d)$, we have
    \[
    \Vert u\Vert_{p^{*}}\leq C_p\Vert u\Vert_{D^{1,p}}
    \]
    where $1\leq p\textless d$, $p^{*}=\frac{dp}{d-p}$, $\Vert u\Vert_{D^{1,p}}=(\int_{\mathbb{Z}^d}\vert\nabla u\vert^pdx)^{\frac{1}{p}}$ and $D^{1,p}(\mathbb{Z}^d)$ is the completion of $C_0(\mathbb{Z}^d)$ under the $D^{1,p}$ norm.
\end{theorem}
Using Theorem 4,6, we have the following theorem.
\begin{theorem}\label{ourGNS}

      Suppose $d\geq 3$ or $d=1$, if $2\textless p\leq p_d$ with 
    
    \begin{equation*}
p_d:=\left\{
\begin{aligned}
    &\infty,d=1\\
    &\frac{2d}{d-2},d\geq 3
\end{aligned}
\right.
\end{equation*}
    then for every $\theta=\min(1,d(\frac{1}{2}-\frac{1}{p}))$, there exists constant $C_{d,p}$ such that the following inequality holds for any function $u\in l^2(\mathbb{Z}^d)$
\begin{equation*}
\Vert u\Vert_p\leq C_{d,p}\Vert\nabla u\Vert_2^{\theta}\Vert u\Vert_2^{1-\theta}.   
\end{equation*}
Suppose $d=2$, then there exists constant $C$ such that the following inequality holds for any function $u\in l^2(\mathbb{Z}^2)$
\[
\Vert u\Vert_4^4\leq C\Vert\nabla u\Vert_2^{2}\Vert u\Vert_2^{2}.    
\]
\end{theorem}
\begin{proof}
    If $d\geq 3$ and $2\textless p\leq p_d$, by Theorem \ref{sobolev} and H\"{o}lder's inequality, we have
    \[
        \Vert u\Vert_p\leq C\Vert u\Vert_{p_d}^{d(\frac{1}{2}-\frac{1}{p})}\Vert u\Vert_2^{1-d(\frac{1}{2}-\frac{1}{p})}\leq C\Vert \nabla u\Vert_{2}^{d(\frac{1}{2}-\frac{1}{p})}\Vert u\Vert_2^{1-d(\frac{1}{2}-\frac{1}{p})}.
    \]

    If $d=1$ and $2\textless p\leq\infty$, then for any $v\in l^1(\mathbb{Z})$, suppose $\vert v(x)\vert=\Vert v\Vert_{\infty}$ for some $x$, then we have 
    \begin{align*}
        \vert v(x)\vert&\leq\sum\limits_{i=0}^{k}\vert v(x+i)-v(x+i+1)\vert +\vert v(x+k+1)\vert\\
        &\leq C\Vert\nabla v\Vert_{1} +\vert v(x+k+1)\vert.
    \end{align*}
    Letting $k\rightarrow\infty$, we have that
    \begin{equation}
    \Vert v\Vert_{\infty}\leq C\Vert \nabla v\Vert_{1} .   
    \end{equation}
    
    For any $u\in l^2(\mathbb{Z})$, substituting $v = u ^2$ into (9), we have that
    \begin{align*}
        \Vert u\Vert^{2}_{\infty}&\leq C\sum\limits_{x\sim y}\vert u^{2}(x)-u^{2}(y)\vert\\
        &\leq C\sum\limits_{x\sim y}\vert u(x)-u(y)\vert(\vert u(x)\vert+\vert u(y)\vert)\\
        &\leq  C(\sum\limits_{x\sim y}\vert u(x)-u(y)\vert^{2})^{\frac{1}{2}}(\sum\limits_{x\sim y}\vert u(x)\vert^{2}+\vert u(y)\vert^{2})^{\frac{1}{2}}\\
        &\leq C\Vert \nabla u\Vert_{2}\Vert u\Vert_{2}.
    \end{align*}
    Hence we have
    \begin{align*}
        \Vert u\Vert_p^p&\leq \Vert u\Vert_2^2\Vert u\Vert_{\infty}^{p-2}\\
        &\leq C\Vert u\Vert_2^2(\Vert\nabla u\Vert_2\Vert u\Vert_{2})^{\frac{p-2}{2}}\\
        &=C\Vert \nabla u\Vert_2^{\frac{p-2}{2}}\Vert u\Vert_{2}^{\frac{p+2}{2}}.
    \end{align*}
    
    If $d=2$, by Theorem \ref{sobolev}, for all $v\in l^1(\mathbb{Z}^2)$, we have
    \[
    \Vert v\Vert_2\leq C\Vert \nabla v\Vert_1.
    \]
    Hence, for all $u\in l^2(\mathbb{Z}^2)$, taking $v=u^2$, we have
    \begin{align*}
        \Vert u\Vert^{2}_{4}&\leq C\sum\limits_{x\sim y}\vert u^{2}(x)-u^{2}(y)\vert\\
        &\leq C\sum\limits_{x\sim y}\vert u(x)-u(y)\vert(\vert u(x)\vert+\vert u(y)\vert)\\
        &\leq  C(\sum\limits_{x\sim y}\vert u(x)-u(y)\vert^{2})^{\frac{1}{2}}(\sum\limits_{x\sim y}\vert u(x)\vert^{2}+\vert u(y)\vert^{2})^{\frac{1}{2}}\\
        &\leq C\Vert\nabla u\Vert_{2}\Vert u\Vert_{2}.
    \end{align*}
\end{proof}
In conclusion, we have the following corollary.
\begin{corollary}\label{masscritiGNS}
    There exists constant $C_{d,2+\frac{4}{d}}$ such that for any function $u\in l^2(\mathbb{Z}^d)$, $d\geq1$, the following inequality holds
    \[
    \Vert u\Vert_{2+\frac{4}{d}}^{2+\frac{4}{d}}\leq C_{d,2+\frac{4}{d}}\Vert \nabla u\Vert^2_2\Vert u\Vert_2^{\frac{4}{d}}.
    \]
\end{corollary}
Using Corollary \ref{masscritiGNS}, we can establish Theorem \ref{thm2} in the special case that $V(x)\equiv 0$ and $f(x,s)=\tilde{f}(s)$.
\begin{theorem}\label{limitcritirion}
    (i)If $\varliminf\limits_{s\rightarrow0}\frac{\tilde{F}(s)}{\vert s\vert^{2+\frac{4}{d}}}=\infty$ , then $\alpha^{\infty}=0$.\\ 
    (ii)If $\varlimsup\limits_{s\rightarrow0}\frac{\tilde{F}(s)}{\vert s\vert^{2+\frac{4}{d}}}\textless\infty$, then $\alpha^{\infty}\textgreater0$.
\end{theorem}
\begin{proof}

(i) For arbitrary $a\textgreater0$, let
    \[
    u_{n}(x)=\left\{
    \begin{aligned}
       & c_{a,d,n}\frac{n-\vert x\vert}{n^{\frac{d}{2}+1}}, \vert x\vert\leq n\\
       & 0, \mathrm{otherwise}
    \end{aligned}
    \right.
    \]
    where $c_{a,d,n}$ is taken such that $\Vert u_n\Vert_2^2=a.$ Then we have $\Vert\nabla u_n\Vert_2^2\leq C_{a,d}^1n^{-2}$ and $\Vert u_n\Vert_p^p\leq C^2_{a,d}n^{-\frac{d(p-2)}{2}}$ for all $2\textless p\textless \infty$, therefore $\lim\limits_{n\rightarrow\infty}\Vert u_n\Vert_{\infty}=0$. Since $\varliminf_{s\rightarrow0}\frac{\tilde{F}(s)}{\vert s\vert^{2+\frac{4}{d}}}=\infty$, for any large $\delta\textgreater0$, there exists $\eta\textgreater0$ such that
    \[
    \tilde{F}(s)\geq \delta \vert s\vert^{2+\frac{4}{d}}
    \]
    holds for all $s\in[-\eta,\eta]$. We take sufficiently large $\delta$  such that $\frac{1}{2}C^1_{a,d} -C^2_{a,d}\delta\textless0 $. Choosing large $n$ such that $\Vert u_n\Vert_{\infty}\leq\eta$, we have that
     \[
    \begin{aligned}
        \Phi^{\infty}(u_n)&\leq \frac{1}{2}\int_{\mathbb{Z}^d}\vert\nabla u_n\vert^2dx-\delta\int_{\mathbb{Z}^d}\vert u_n\vert^{2+\frac{4}{d}}dx\\
        &\leq \frac{1}{2}C^1_{a,d} n^{-2}-C^2_{a,d}\delta n^{-2}.\\
    \end{aligned}
    \]
     Then $\Phi^{\infty}(u_n)\textless 0$, and we have $E_a^{\infty}\textless 0$.\\

(ii)If $\varlimsup\limits_{s\rightarrow0}\frac{\tilde{F}(s)}{\vert s\vert^{2+\frac{4}{d}}}\textless\infty$, then there exist constants $C_{F}$ and $\eta\textgreater0$ such that 
\[
\tilde{F}(s)\leq C_{F} \vert s\vert^{2+\frac{4}{d}}
\]
holds for all $s\in[-\eta,\eta].$ Hence for $a\leq \eta^2$ and any $u\in S_a$, we have $\Vert u\Vert_{\infty}\leq\Vert u\Vert_2\leq\eta$. By Corollary \ref{masscritiGNS}, we have
\[
\begin{aligned}
    \Phi^{\infty}(u)&\geq \frac{1}{2}\int_{\mathbb{Z}^d}\vert\nabla u\vert^2dx-C_{F}\int_{\mathbb{Z}^d}\vert u\vert^{2+\frac{4}{d}}dx\\
    &\geq \frac{1}{2}\int_{\mathbb{Z}^d}\vert\nabla u\vert^2dx-C_{F}C_{d,2+\frac{4}{d}}a^{\frac{2}{d}}\Vert\nabla u\Vert_2^{2}.\\
\end{aligned}
\]
Hence, for $a\leq\min(\eta^2,(\frac{1}{2C_FC_{d,2+\frac{4}{d}}})^{\frac{d}{2}})$, we have that $\Phi^{\infty}(u)\geq 0 $ holds for all $u\in S_a$. Therefore, $E_a^{\infty}=0$ and $\alpha^{\infty}\textgreater 0$.
\end{proof}
By the above results, we have prove the limit case. We now consider the general case. We assume there exists $x\in\mathbb{Z}^d$ such that $V(x)\textless V_{\infty}=0$ or $f(x,s)\textgreater\tilde{f}(s)$, otherwise it is reduced to the limit case. Then we have the following lemmas, which will be useful to prove the main result.
\begin{lemma}\label{comparelimit}
    (i) $\forall a\textgreater0$, we have $E_a\leq E_a^{\infty}$.\\
    (ii) If $E_a^{\infty}$ is attained, then we have $E_a\textless E_a^{\infty}$.
\end{lemma}
\begin{proof}

    (i) Since $V(x)\leq 0$ and $f(x,s)\geq\tilde{f}(s)$, it is obvious that we have $E_a\leq E_a^{\infty}$.\\
    
    (ii) Suppose $u$ is a global minimizer of $\Phi^{\infty}$ with respect to $E_a^{\infty}$, let $x_1\in\mathbb{Z}^d$ such that  $V(x_1)\textless V_{\infty}=0$ or $f(x_1,s)\textgreater\tilde{f}(s)$. Take $x_0$ such that $u(x_0)\neq 0.$ Setting $\tilde{u}(x)=u(x+x_0-x_1)$,  we have
    \[
    \begin{aligned}
        E_a&\leq \frac{1}{2}\int_{\mathbb{Z}^d}\vert\nabla\tilde{u}\vert^2+V(x)\tilde{u}^2dx-\int_{\mathbb{Z}^d}F(x,\tilde{u})dx\\
        &=\frac{1}{2}\int_{\mathbb{Z}^d}\vert\nabla\tilde{u}\vert^2dx-\int_{\mathbb{Z}^d}\tilde{F}(\tilde{u})dx+\frac{1}{2}\int_{\mathbb{Z}^d}V(x)\tilde{u}^2dx-\int_{\mathbb{Z}^d}(F(x,\tilde{u})-\tilde{F}(\tilde{u}))dx\\
        &=E^{\infty}_a+\frac{1}{2}\int_{\mathbb{Z}^d}V(x)\tilde{u}^2dx-\int_{\mathbb{Z}^d}(F(x,\tilde{u})-\tilde{F}(\tilde{u}))dx.
    \end{aligned}
    \]
    Since $\frac{1}{2}\int_{\mathbb{Z}^d}V(x)\tilde{u}^2dx\leq 0$ and $-\int_{\mathbb{Z}^d}(F(x,\tilde{u})-\tilde{F}(\tilde{u}))dx\leq0$ can not vanish simultaneously, we have
    \[
    E_a\textless E_a^{\infty}.
    \]
\end{proof}
\begin{lemma}\label{compareattained}
    If $E_a\textless E_a^{\infty}$, then $E_a$ is attained.
\end{lemma}
\begin{proof}

    Take a minimizing sequence $\{u_n\}\subset S_a$ with respect to $E_a$, then $\{u_n\}$ is bounded in $l^2(\mathbb{Z}^d)$ by Lemma \ref{boundness}, so that taking a subsequence if necessary we assume $u_n\rightharpoonup u$ in $l^2(\mathbb{Z}^d)$. Lemma \ref{PhiBreLieb} implies that
    \[
    \Phi(u)=\Phi(u_n)+\Phi(u_n-u)+o(1).
    \]
    Assume $\Vert u\Vert_{2}^2=\eta$, then $\eta\leq a$. We prove that $\eta=a$.\\
    
Suppose that this is not true. If $\eta=0$, then for all $x\in\mathbb{Z}^d$, we have $\lim\limits_{n\rightarrow\infty}u_n(x)=0$. Since $\lim\limits_{\vert x\vert\rightarrow\infty}V(x)=0$, 
    by Lemma \ref{Vvanish} we have $\lim\limits_{n\rightarrow\infty}\int_{\mathbb{Z}^d}V(x)u_n^2dx=0$. 
    For any $\epsilon\textgreater0$, taking $\delta\textgreater0$ and $R\textgreater0$ which will be determined later, we have following estimate
    \[
    \begin{aligned}
        \vert\int_{\mathbb{Z}^d}(F(x,u_n)-\tilde{F}(u_n))dx\vert\leq&\int_{B_R}\vert F(x,u_n)-\tilde{F}(u_n)\vert dx\\&
        +\int_{x\in B_R^c,\vert u_n\vert\leq\delta}\vert F(x,u_n)-\tilde{F}(u_n)\vert dx\\
        &+\int_{x\in B_R^c,\delta\leq\vert u_n\vert\leq\frac{1}{\delta}}\vert F(x,u_n)-\tilde{F}(u_n)\vert dx\\
        &+\int_{x\in B_R^c,\frac{1}{\delta}\leq\vert u_n\vert}\vert F(x,u_n)-\tilde{F}(u_n)\vert dx\\
        &\triangleq I_1+I_2+I_3+I_4.
    \end{aligned}
    \]
    Now we estimate $I_1,I_2,I_3,I_4$ seperately. \\
    
    For $I_1$, since for any $x\in\mathbb{Z}^d$, $\lim\limits_{n\rightarrow\infty}u_n(x)=0$, we have 
    \begin{align*}
    &\int_{B_R}\vert F(x,u_n)-\tilde{F}(u_n)\vert dx\\
    &\leq \int_{B_R}\vert F(x,u_n)\vert dx+\int_{B_R}\vert\tilde{F}(u_n)\vert dx\\
    &\leq \epsilon    
    \end{align*}
    for sufficiently large $n.$ \\
    
    For $I_2$, from condition $(f1)$ we have 
    \begin{align*}
    &\int_{x\in B_R^c,\vert u_n\vert\leq\delta}\vert F(x,u_n)-\tilde{F}(u_n)\vert dx\\
    &\leq \sup_{x\in\mathbb{Z}^d,0\textless s\leq\delta}\frac{\vert F(x,s)-\tilde{F}(s)\vert}{s^2} \int_{x\in B_R^c}\vert u(x)\vert^2 dx\\
    &\leq a\epsilon    
    \end{align*}
for sufficiently small $\delta$.\\
    
    For $I_3$, from condition $(f2)$ we have
    \[
    \begin{aligned}
        &\int_{x\in B_R^c,\delta\leq\vert u_n\vert\leq\frac{1}{\delta}}\vert F(x,u_n)-\tilde{F}(u_n)\vert dx\\
        &\leq\int_{x\in B_R^c,\delta\leq\vert u_n\vert\leq\frac{1}{\delta}}\vert F(x,u_n)-\tilde{F}(u_n)\vert\frac{\vert u(x)\vert^2}{\delta^2} dx\\
        &\leq\frac{a\epsilon}{\delta^2}
    \end{aligned}
    \]
    for sufficiently large $R$.\\
    
    For $I_4$, from condition $(f1)$ we have
    \begin{align*}
            &\int_{x\in B_R^c,\frac{1}{\delta}\leq\vert u_n\vert}\vert F(x,u_n)-\tilde{F}(u_n)\vert dx\\
            &
    \leq \sup_{x\in\mathbb{Z}^d,s\geq \frac{1}{\delta}}\frac{\vert F(x,s)-\tilde{F}(s)\vert}{s^{q+1}}\int_{x\in B_R^c}\vert u_n\vert^{q+1} dx\\
    &\leq C\epsilon
    \end{align*}
    for sufficiently small $\delta$.
    In conclusion, we have 
    \[
    \lim_{n\rightarrow\infty}\int_{\mathbb{Z}^d}(F(x,u_n)-\tilde{F}(u_n))dx=0.
    \]
    Hence, we have
    \[
    \begin{aligned}
        E_a&=\Phi(u_n)+o(1)\\
        &=\frac{1}{2}\int_{\mathbb{Z}^d}\vert\nabla u_n\vert^2+V(x)u_n^2dx-\int_{\mathbb{Z}^d}F(x,u_n)dx+o(1)\\
        &=\frac{1}{2}\int_{\mathbb{Z}^d}\vert\nabla u_n\vert^2dx-\int_{\mathbb{Z}^d}\tilde{F}(u_n)dx+o(1)\\
        &\geq E_a^{\infty} + o(1).
    \end{aligned}
    \]\
    But this is a contradiction to $E_a\textless E_a^{\infty}$.\\
    
    If $0\textless \eta\textless a$, then by Corollary \ref{NormBreLieb} we have $\Vert v_n\Vert_{2}^2=a-\eta+o(1)$. If $E_{\eta}$ can not be attained, then we have
    \[
    \begin{aligned}
        E_{a}&=\Phi(u)+\Phi(v_n)+o(1)\\
        &\geq \Phi(u)+E_{a-\eta}+o(1)
    \end{aligned}
    \]
   Letting $n\rightarrow\infty,$ we have
    \[
    \begin{aligned}
        E_{a}&\geq \Phi(u)+E_{a-\eta}\\
        &\textgreater E_{\eta}+E_{a-\eta}\\
        &\geq E_a.
    \end{aligned}
    \]
    This is a contradiction. Hence $u$ is a global minimizer of $\Phi$ with respect to $E_{\eta}$, but then 
    \[
    E_a\geq E_{\eta}+E_{a-\eta}.
    \]
    This is a contradiction to $E_a\textless E_{\eta}+E_{a-\eta}$ since $E_{\eta}$ is attained.\\
    
    In conclusion, $\Vert u\Vert_2^2=a$. By Corollary \ref{NormBreLieb} we have $\lim\limits_{n\rightarrow\infty}\Vert u_n-u\Vert_2=0$, hence by Lemma \ref{Fvanish}, $\lim\limits_{n\rightarrow\infty}\int_{\mathbb{Z}^d}F(x,u_n-u)dx=0.$ Note that $$\lim\limits_{n\rightarrow\infty}\int_{\mathbb{Z}^d}\vert\nabla(u_n-u)\vert^2dx\leq C\lim\limits_{n\rightarrow\infty}\int_{\mathbb{Z}^d}\vert u_n-u\vert^2dx=0.$$ Using $\lim\limits_{n\rightarrow\infty}\Vert u_n-u\Vert_{\infty}\leq \lim\limits_{n\rightarrow\infty}\Vert u_n-u\Vert_2=0,$ by Lemma \ref{Vvanish} we have $\lim\limits_{n\rightarrow\infty}\int_{\mathbb{Z^d}}V(x)\vert u_n-u\vert^2dx=0$. In conclusion $\lim\limits_{n\rightarrow\infty}\Phi(u_n-u)=0$, which implies $\Phi(u)=\lim\limits_{n\rightarrow\infty}\Phi(u_n)=E_a$ by Lemma \ref{PhiBreLieb}, therefore $u$ is a global minimizer of $\Phi$ with respect to $E_a$.
\end{proof}
At the end, we are ready to prove Theorem \ref{thm1}
\begin{proof}[{Proof of  Theorem \ref{thm1}}]
    Take $\alpha=\inf\{a\enspace|E_a\textless0\}.$ If $a\textgreater\alpha$ and $E_a^{\infty}=0$, then $E_a\textless 0=E_a^{\infty}.$ By Lemma \ref{compareattained} $E_a$ is attained. If $a\textgreater\alpha$ and $E_a^{\infty}\textless0$, then by Theorem \ref{limitmain}, $E_a^{\infty}$ can be attained. Hence $E_a\textless E_a^{\infty}$ by (ii) in Lemma \ref{comparelimit}, and $E_a$ can be attained by Lemma \ref{compareattained}. If $0\textless a\textless\alpha$, the proof is similar to the proof of Theorem \ref{limitmain}.
\end{proof}

Before the proof of Theorem \ref{thm2}, we recall the Hardy inequality on the lattice $\mathbb{Z}^d$, $d\geq 3$; see \citep{MR2544038, MR3774437}.

\begin{theorem}\label{hardy}
    If $d\geq 3$, there exists constant $C_d$ such that
    \[
    C_d\int_{\mathbb{Z}^d}\frac{u^2}{1+\vert x\vert^2}dx\leq \int_{\mathbb{Z}^d}\vert\nabla u\vert^2dx,u\in C_c(\mathbb{Z}^d).
    \]
\end{theorem}
Using Theorem \ref{hardy}, we can prove Theorem \ref{thm2}

\begin{proof}[{Proof of  Theorem \ref{thm2}}]
    Without loss of generality, we can assume $V_{\infty}=0$.\\

    (i) Considering the function $u_R$ in the proof of Lemma \ref{funproper}, we have 
    \[
    \Phi(u_R)\leq d\xi^2(2R+1)^{d-1}-\tilde{F}(\xi)(2R+1)^d.
    \]
    Take $R=\frac{1}{2}[\frac{d\xi^2}{\tilde{F}(\xi)}]$, then $\Phi(u_R)\textless0$, and $\Vert u_R\Vert_2^2=\xi^2([\frac{d\xi^2}{\tilde{F}(\xi)}]+1)^{d}$. Hence, $\alpha\textless\xi^2([\frac{d\xi^2}{\tilde{F}(\xi)}]+1)^{d}$.\\
    
    (ii) This is a direct consequence of Theorem \ref{limitcritirion} and Lemma \ref{comparelimit}.\\
    
    (iii) By Theorem \ref{hardy}, we have 
    \[
    \frac{1}{2}\int_{\mathbb{Z}^d}V(x)u^2dx\geq - \frac{C_d}{2}(1-\epsilon)\int_{\mathbb{Z}^d}\frac{1}{1+\vert x\vert^2}u^2dx\geq -(\frac{1-\epsilon}{2})\int_{\mathbb{Z}^d}\vert\nabla u\vert^2dx.
    \]
    Then the proof for $\alpha\textgreater0$ is similar to the proof of Theorem \ref{limitcritirion}. For the lower bound, suppose $u\in l^2(\mathbb{Z}^d)$ such that $\Vert u\Vert_2^2\leq \delta^2$, then $\Vert u\Vert_{\infty}\leq \delta$. We have
    \begin{align*}
    \Phi(u)&\geq\frac{\epsilon}{2}\int_{\mathbb{Z}^d}\vert\nabla u\vert^2dx-\int_{\mathbb{Z}^d}F(x,u)dx  \\
    &\geq\frac{\epsilon}{2}\int_{\mathbb{Z}^d}\vert\nabla u\vert^2dx- C_F\int_{\mathbb{Z}^d}\vert u\vert^{2+\frac{4}{d}}dx\\
    &\geq\frac{\epsilon}{2}\int_{\mathbb{Z}^d}\vert\nabla u\vert^2dx- C_FC_{d,2+\frac{4}{d}}a^{\frac{2}{d}}\int_{\mathbb{Z}^d}\vert\nabla u\vert^2dx.
    \end{align*}
    Hence, for $a\leq\min((\frac{\epsilon}{2C_{F}C_{d,2+\frac{4}{d}}})^{\frac{d}{2}},\delta^2)$, $\Phi(u)\geq 0 $.\\
    This proves the theorem.
    
\end{proof}

\textbf{Acknowledgements.}
The author is deeply grateful to Professor Bobo Hua for helpful discussions and constant support.


\end{document}